\documentclass[10pt]{amsart}
\usepackage{epsfig,url}
\usepackage{mathdots}

\newtheorem{theorem}{Theorem}[section]
\newtheorem{lemma}[theorem]{Lemma}

\newtheorem{proposition}[theorem]{Proposition}

\theoremstyle{definition}

\newtheorem{example}[theorem]{Example}

\newtheorem{note}[theorem]{Note}

\theoremstyle{remark}

\begin{document}

\title[A symbolic approach to {B}ernoulli-{B}arnes polynomials]{A symbolic approach to 
some identities for {B}ernoulli-{B}arnes polynomials}

\author[]{Lin Jiu}
\address{Department of Mathematics,
Tulane University, New Orleans, LA 70118}
\email{ljiu@tulane.edu}

\author[]{Victor H. Moll}
\address{Department of Mathematics,
Tulane University, New Orleans, LA 70118}
\email{vhm@tulane.edu}

\author[]{Christophe Vignat}
\address{Department of Mathematics,
Tulane University, New Orleans, LA 70118}
\email{cvignat@tulane.edu}

%    General info

\date{\today}


\begin{abstract}
A symbolic method is used to establish some properties of the Bernoulli-Barnes polynomials. 
\end{abstract}

\maketitle

\newcommand{\ba}{\begin{eqnarray}}
\newcommand{\ea}{\end{eqnarray}}
\newcommand{\ift}{\int_{0}^{\infty}}
\newcommand{\nn}{\nonumber}
\newcommand{\no}{\noindent}
\newcommand{\lf}{\left\lfloor}
\newcommand{\rf}{\right\rfloor}
\newcommand{\realpart}{\mathop{\rm Re}\nolimits}
\newcommand{\imagpart}{\mathop{\rm Im}\nolimits}

\newcommand{\op}[1]{\ensuremath{\operatorname{#1}}}
\newcommand{\pFq}[5]{\ensuremath{{}_{#1}F_{#2} \left( \genfrac{}{}{0pt}{}{#3}
{#4} \bigg| {#5} \right)}}

\newtheorem{Definition}{\bf Definition}[section]
\newtheorem{Thm}[Definition]{\bf Theorem}
\newtheorem{Example}[Definition]{\bf Example}
\newtheorem{Lem}[Definition]{\bf Lemma}
\newtheorem{Cor}[Definition]{\bf Corollary}
\newtheorem{Prop}[Definition]{\bf Proposition}
\numberwithin{equation}{section}

\section{Introduction}
\label{sec-intro}

The Bernoulli numbers $B_{n}$, defined by their exponential generating function
\begin{equation}
\frac{z}{e^{z}-1} = \sum_{k=0}^{\infty} B_{k} \frac{z^{k}}{k!}
\end{equation}
\noindent
have produced  a variety of generalizations in the literature. The so-called Bernoulli-Barnes 
numbers $B_{k}(\mathbf{a})$, defined by 
\begin{equation}
\prod_{j=1}^{n} \frac{z}{e^{a_{j}z} - 1} = \sum_{k=0}^{\infty} \frac{B_{k}(\mathbf{a})}{k!} z^{k},
\end{equation}
depend on a multi-dimensional parameter $\mathbf{a} = (a_{1}, \cdots, a_{n}) \in \mathbb{C}^{n}$.  The Bernoulli numbers 
correspond to $n=1$ and $\mathbf{a}=1$.

For any sequence of numbers $\{ a_{j} \}$  with exponential generation function 
$\begin{displaystyle}  
  f(z) = \sum_{j=0}^{\infty} a_{j} \frac{z^{j}}{j!},
  \end{displaystyle}$
 associate the  sequence of polynomials
 $ \begin{displaystyle}
  A_{j}(x) = \sum_{\ell=0}^{j} \binom{j}{\ell} a_{j-\ell}x^{\ell}.
  \end{displaystyle}$
 An elementary argument shows that $e^{xz} f(z)$ is the exponential generating function for $\{ A_{j}(x) \}$.  This 
 produces, from $B_{k}(\mathbf{a})$, the    {\textit{Bernoulli-Barnes polynomials}}
\begin{equation}
B_{j}(x;\mathbf{a}) = \sum_{\ell=0}^{j}  \binom{j}{\ell} B_{j-\ell}(\mathbf{a}) x^{\ell}
\end{equation}
with exponential generating function
\begin{equation}
\sum_{j=0}^{\infty} B_{j}(x;\mathbf{a}) \frac{z^{j}}{j!} = e^{xz} \prod_{k=1}^{n} \frac{z}{e^{a_{k}z} - 1}.
\end{equation}

In the  special case $\mathbf{1} = (1, \cdots, 1)$ one obtains the N\"{o}rlund polynomials  $B_{j}(x;\mathbf{1})$
\begin{equation}
\sum_{j=0}^{\infty} B_{j}(x;\mathbf{1}) \frac{z^{j}}{j!} = e^{xz}  \frac{z^{n}}{(e^{z}-1)^{n}}.
\label{gen-norlund1}
\end{equation}

The Bernoulli-Barnes numbers $B_{k}(\mathbf{a})$ can be expressed  in terms of the Bernoulli numbers $B_{k}$  by
the multiple sum 
\begin{equation}
B_{k}(\mathbf{a}) = \sum_{m_{1} + \cdots + m_{n} = k} 
\binom{k}{m_{1}, \cdots, m_{n}} a_{1}^{m_{1}-1}  \cdots a_{n}^{m_{n}-1} B_{m_{1}} \cdots B_{m_{n}}.
\end{equation}
\noindent
 Therefore $a_{1}\cdots a_{n} B_{k}(\mathbf{a})$ is also a  polynomial in $\mathbf{a}$. Some parts of the literature
 refer to them as the 
 Bernoulli-Barnes polynomials. The reader should be  aware of this share of nomenclature.  \\

The first result requires the notion of a self-dual  sequence. Recall that $\{ a_{n} \}$ is called self-dual if it satisfies
\begin{equation}
a_{n}  = \sum_{k=0}^{n} \binom{n}{k} (-1)^{k} a_{k}, \quad \text{ for all } n \in \mathbb{N}.
\end{equation}
\noindent
The recent study \cite{bayad-2014a} contains the following statement  as Corollary $5.4$:\\

{\textit{
 Let $\mathbf{a} = (a_{1}, \cdots, a_{n}) \in \mathbb{C}^{n}$ with 
 $A = a_{1} + \cdots + a_{n} \neq 0$. Then
the sequence  $\{ (-1)^{n} A^{-n} B_{n}(\mathbf{a}) : \, n \in \mathbb{Z}_{\geq 0} \}$ is a self-dual sequence.
}}

\smallskip

The authors state that 
\begin{center}
{\textit{It would be interesting to prove this statement  directly}}.
\end{center}

\smallskip

Section \ref{sec-selfdual} describes 
self-dual sequences and provides the requested direct proof.

\medskip

The arguments presented here are in the spirit of symbolic calculus. In this framework, one 
defines a  {\textit{Bernoulli symbol}} $\mathcal{B}$  and an {\textit{evaluation map}} $eval$ such that 
\begin{equation}
eval \left( \mathcal{B}^{n} \right) = B_{n}.
\end{equation}
\noindent
The reader is referred to \cite{dixit-2014a} and \cite{gessel-2003a} for the rules of this method. To illustrate the main idea, and 
omitting the {\it{eval }} operator to simplify notation, consider the symbolic identity
\begin{equation}
e^{\mathcal{B}z} = \frac{z}{e^{z}-1}.
\end{equation}
\noindent
This is explained  by the identities 
\begin{equation}
eval \left( e^{\mathcal{B}z } \right) = 
eval \left( \sum_{n =0} ^{\infty} \frac{\mathcal{B}^{n}}{n!} z^{n} \right) = 
\sum_{n = 0} ^{\infty} \frac{eval \left( \mathcal{B}^{n} \right)}{n!} z^{n}  = 
\sum_{n = 0} ^{\infty} \frac{B_{n} z^{n}}{n!}  = \frac{z}{e^{z}-1}.
\end{equation}
\noindent
The symbolic version of the Bernoulli polynomials $B_{n}(x)$, defined by the generating  function
\begin{equation}
\sum_{n = 0}^{\infty}  \frac{B_{n}(x)}{n!} z^{n} = \frac{ze^{xz}}{e^{z}-1}
\end{equation}
\noindent
is simply (where the {\textit{eval}} map has been omitted again)
\begin{equation}
B_{n}(x) = ( \mathcal{B} + x)^{n}.
\end{equation}
The principle of symbolic calculus is to perform all computations replacing the Bernoulli polynomial $B_{n}(x)$ by 
the symbol $(\mathcal{B}+ x)^{n}$ and, at the end of the process, apply the evaluation map to obtain the result. The 
basic expression for Bernoulli polynomials in terms of Bernoulli numbers  illustrates the method: 
\begin{equation}
B_{n}(x) = (\mathcal{B}+x)^{n} = \sum_{k=0}^{n} \binom{n}{k} \mathcal{B}^{k} x^{n-k} = 
 \sum_{k=0}^{n} \binom{n}{k} B_{k}x^{n-k}.
 \end{equation}
 The symbolic representation of the Bernoulli-Barnes numbers is obtained from a collection of $n$ independent 
 Bernoulli symbols 
 $\{ \mathcal{B}_{i} \}_{1 \leq i \leq n }$, where independence is understood in the sense that 
 \begin{equation}
 e^{z( \mathcal{B}_{i} + \mathcal{B}_{j})}  = e^{z \mathcal{B}_{i}}  e^{z \mathcal{B}_{j}}, \text{ for any } i \neq j.
 \end{equation}
\noindent 
Then the Bernoulli-Barnes numbers $B_{k}(\mathbf{a})$ are given  in terms $\mathbf{a} = (a_{1}, \cdots, a_{n})$ and 
$\mathcal{B} = ( \mathcal{B}_{1}, \cdots, \mathcal{B}_{n})$ by
\begin{equation}
B_{k}(\mathbf{a}) = \frac{1}{|\mathbf{a}|} \left( \mathbf{a} \cdot \mathcal{B} \right)^{k}
\end{equation}
\noindent
where 
\begin{equation}
\mathbf{a} \cdot \mathcal{B} = \sum_{k=1}^{n} a_{k} \mathcal{B}_{k} \text{ and } 
|\mathbf{a}| = \prod_{k=1}^{n} a_{k}.
\end{equation}
\noindent
Similarly, the Bernoulli-Barnes polynomials are represented symbolically by 
\begin{equation}
B_{k}(\mathbf{a};x) = \frac{1}{|\mathbf{a}|} ( x + \mathbf{a} \cdot \mathcal{B} )^{k}.
\end{equation}

 \section{A difference formula}
 \label{sec-difference}
 
 The section in \cite{bayad-2014a} containing the requested proof begins with a difference formula for the 
 Bernoulli-Barnes polynomials.   A direct proof by symbolic arguments is presented here. For 
 any $L \subset  \{ 1, \cdots, n \}$, say $L = \{ i_{1}, \cdots, i_{r} \}$, introduce the notation 
 \begin{equation}
 {\mathbf{a}}_{L} = (a_{i_{1}}, \cdots, a_{i_{r}}). 
 \end{equation}
 In general, any symbol with a set  $L \subset \{ 1, \cdots, n \}$ as a subscript, indicates that the indices appearing in the 
  symbol should be  restricted to those in the set $L$. For instance,  ${\mathbf{a}}_{\{2, 5 \}} = (a_{2}, a_{5})$ and 
  $|\mathbf{a}|_{\{ 2, 5  \} }= a_{2}a_{5}$. 
 
 \medskip
 
 \noindent
 Theorem $5.1$ in \cite{bayad-2014a} is restated here.
 
 \begin{theorem}
 \label{thm-51}
 For $\mathbf{a} = (a_{1}, \cdots, a_{n}) \in \mathbb{C}^{n}$ and  
 $\begin{displaystyle} A = \sum_{i=1}^{n} a_{i} \end{displaystyle}$, we have the difference formula 
 \begin{equation}
 (-1)^{m} B_{m}(-x;\mathbf{a}) - B_{m}(x;\mathbf{a}) = m! \sum_{\ell=0}^{n-1} \sum_{|L| = \ell} 
 \frac{B_{m-n+\ell}(x; {\mathbf{a}}_{L})}{(m-n+\ell)!}
 \label{diff-51}
 \end{equation}
 \noindent
 with $B_{m}(x; {\mathbf{a}}_{L}) = x^{m}$ if $L = \emptyset$. Furthermore, 
 \begin{equation}
 B_{m} \left( x + A; \mathbf{a} \right) = (-1)^{m} B_{m}(-x; \mathbf{a} ).
 \end{equation}
 \end{theorem}
 
 It is shown that Theorem \ref{thm-51} is a special case of a general expansion formula. A variety of proofs are presented 
 below.  The conditions imposed on the function $f$ in the statement of Theorem \ref{thm-general} are 
 those required for the existence of the expressions appearing in it. Those functions will be called {\textit{reasonable}}. In 
 particular polynomials are reasonable functions.  Here $f^{(j)}(x)$ represents the $j$-th derivative of $f$.
 
 \begin{theorem}
 \label{thm-general}
 Let $f$ be a reasonable function. Then, with 
 $\mathbf{a} = (a_{1}, \cdots, a_{n})$,
 \begin{equation}
 f( x - \mathbf{a} \cdot \mathcal{B})  = \sum_{j=0}^{n} \sum_{|J| = j} |a|_{J^{*}} 
 f^{(n-j)} \left( x + ( \mathbf{a} \cdot \mathcal{B})_{J}  \right)
 \end{equation}
 \noindent
 where $J \subset \{ 1, \cdots, n \}$ and $J^{*} = \{ 1, \cdots, n \}  \setminus J$. Moreover,
 \begin{equation}
 f \left( x + A + \mathbf{a} \cdot \mathcal{B} \right) = f \left( x - \mathbf{a} \cdot \mathcal{B} \right).
 \end{equation}
 \end{theorem}
 
 \begin{example}
 The theorem gives, for $n=2$ and any reasonable function $f$, the relation 
 \begin{eqnarray*}
 f ( x - a_{1} \mathcal{B}_{1} - a_{2} \mathcal{B}_{2} )  & = &   f ( x + a_{1} \mathcal{B}_{1} + a_{2} \mathcal{B}_{2} ) \\
 & + & a_{1} f'(x + a_{2} \mathcal{B}_{2}) + a_{2} f'(x + a_{1} \mathcal{B}_{1})  + a_{1}a_{2}f''(x).
 \end{eqnarray*}F
 \end{example}
 
 \begin{note}
 The classical differentiation formula 
 \begin{equation}
 \left( \frac{d}{dx} \right)^{j} \frac{B_{n}(x)}{n!} = \frac{B_{n-j}(x)}{(n-j)!}
 \end{equation}
 \noindent
 shows that Theorem \ref{thm-51} is  the special case $f(x) = x^{m}/m!$ of Theorem \ref{thm-general}.
 \end{note}
 
 The proof of Theorem \ref{thm-general} uses some basic identities of symbolic calculus. The proofs are presented here 
 for completeness. 
 
 \begin{lemma}
 \label{lemma-g1}
 Let $g$ be a reasonable function. Then 
 \begin{equation}
 g( - \mathcal{B}) = g( \mathcal{B} + 1) = g( \mathcal{B}) + g'(0).
 \end{equation}
 In particular, 
 \begin{equation}
 -\mathcal{B} = \mathcal{B}+1.
 \label{bplus1}
 \end{equation}
 \end{lemma}
 \begin{proof}
 The proof is presented for the monomial $g(x) = x^{k}$, the general case follows by linearity. The exponential generating 
 function of $(- \mathcal{B})^{k}$ is 
 \begin{eqnarray*} 
 \sum_{k \geq 0} \frac{(- \mathcal{B})^{k} z^{k}}{k!} & = & \text{exp}( - \mathcal{B}z )
   = \frac{-z}{e^{-z} -1} = \frac{ze^{z}}{e^{z}-1} \\
  & = & e^{z} e^{\mathcal{B}z} 
   = e^{z( \mathcal{B} + 1)} \\
  & = & \sum_{k \geq 0} \frac{( \mathcal{B}+1)^{k}z^{k}}{k!},
  \end{eqnarray*}
  \noindent
  which proves the first identity.  Now since  $g(x) = x^{k}$ produces $g'(0) = \delta_{k-1}$ (the Kronecker delta), it 
  follows that 
  \begin{equation}
  \sum_{k \geq 0} \frac{\mathcal{B}^{k} z^{k}}{k!} + 
  \sum_{k \geq 0} \delta_{k-1} \frac{z^{k}}{k!} = \frac{z}{e^{z}-1} + z = \frac{ze^{z}}{e^{z}-1} =  
  e^{(\mathcal{B}+1)z} = 
  \sum_{k \geq 0} \frac{( \mathcal{B}+1)^{k} z^{k}}{k!}
  \end{equation}
  \noindent
  proving the second identity.
  \end{proof}
  
  The first proof of Theorem \ref{thm-general} is given next.
  
  \begin{proof}
  Lemma \ref{lemma-g1} applied to $g(\mathcal{B}) = f(x + a \mathcal{B})$ gives 
  \begin{equation}
  f( x - a \mathcal{B}) = f( x + a \mathcal{B}) + a f'(x).
  \end{equation}
  \noindent
  This is the result for $n=1$. The general case is obtained  by a direct induction argument.
  \end{proof}

 \section{An operational Calculus proof}
 \label{sec-operation}
 
 This section presents a proof of Theorem \ref{thm-general} based on the action of the operator $T_{a}$  on a 
 function $f$ by 
 \begin{equation}
 T_{a}[f(x)]  = f(x - a \mathcal{B}). 
 \end{equation}
 \noindent
 Naturally 
 \begin{equation}
 T_{a_{1}} \circ T_{a_{2}} [f(x)] = T_{a_{1}}[f(x - a_{1} \mathcal{B}_{1}) ] = 
 f( x - a_{1}\mathcal{B}_{1} - a_{2}\mathcal{B}_{2} )
 \end{equation}
 \noindent
 showing that $T_{a_{1}}$ and $T_{a_{2}}$ commute with each other.  On the other hand, since 
 $f(x - {\mathbf{a}} \cdot \mathcal{B}) = f(x + {\mathbf{a}} \cdot \mathcal{B}) + af'(x)$,  the operator $T_{a}$ can be 
 formally expressed as
 \begin{equation}
 T_{a} = e^{a \mathcal{B} \frac{\partial}{\partial x}} + a \frac{\partial}{\partial x},
 \end{equation}
 \noindent
 so that $T_{a}$ is the sum of two commuting operators.  The composition rule 
 \begin{eqnarray}
 T_{a_{1}} \circ T_{a_{2}} & = & e^{(a_{1}\mathcal{B}_{1} + a_{2} \mathcal{B}_{2}) \frac{\partial}{\partial x}}  \\
  & + & a_{1} \frac{\partial}{\partial x} e^{a_{2} \mathcal{B}_{2} \frac{\partial }{\partial x}} +
  a_{2} \frac{\partial}{\partial x} e^{a_{1} \mathcal{B}_{1} \frac{\partial }{\partial x}}  \nonumber \\
  & + & a_{1}a_{2} \frac{\partial^{2}}{\partial x^{2}} \nonumber 
  \end{eqnarray}
  \noindent
  gives the result of Theorem \ref{thm-general} for $n=2$.  The general case follows from the identity
  \begin{eqnarray}
  T_{a_{1}} \circ \cdots \circ T_{a_{n}} & = & \prod_{j=1}^{n} \left( e^{a_{j} \mathcal{B}_{j} \frac{\partial }{\partial x}} + 
  a_{j} \frac{\partial }{\partial x} \right) \\
  & = & \sum_{j=0}^{n} \sum_{|J|=j} |\mathbf{a}|_{J^{*}} \frac{\partial^{n-j}}{\partial x^{n-j}} 
  e^{( \mathbf{a} \cdot \mathcal{B})_{J} \, \frac{\partial }{\partial x}}. \nonumber 
  \end{eqnarray}
  
  \section{A new symbol and another  proof}
  \label{sec-uniform}
  
  This section provides a proof of Theorem \ref{thm-general} based on the uniform symbol $\mathcal{U}$ defined by 
  the  relation
  \begin{equation}
  f(x + \mathcal{U}) = \int_{0}^{1} f(x+u) \, du.
  \end{equation}
  \noindent 
  The uniform symbol acts like the inverse of the Bernoulli symbol, in a sense made precise in the next statement. 
  
  \begin{proposition}
  Let  $\mathcal{B}$ and $\mathcal{U}$ be the Bernoulli and uniform symbols, respectively. 
  Then, for any reasonable function $f$, 
  \begin{equation}
  f( x + \mathcal{U} + \mathcal{B}) = f(x).
  \end{equation}
  \noindent
  In particular, the relations 
  \begin{equation}
  g(x+ \mathcal{B}) = h(x) \text{   and   } h(x + \mathcal{U}) =  g(x)
  \end{equation}
  \noindent
  are equivalent.
  \end{proposition}
  \begin{proof}
  The generating function 
  \begin{equation}
  \sum_{n \geq 0} \frac{(x + \mathcal{U}+ \mathcal{B})^{n}}{n!} z^{n} = 
 e^{zx + z \mathcal{U}+ z \mathcal{B}} = e^{zx} \frac{z}{e^{z}-1} \frac{e^{z}-1}{z} = e^{zx}
 \end{equation}
 \noindent
 shows that $(z + \mathcal{U} + \mathcal{B})^{n} = z^{n}$. The result extends to a general function $g$ by linearity.
  \end{proof}
  
  An interpretation of the special case $n=1$ in Theorem \ref{thm-general} is provided next. This is 
  \begin{equation}
  f(x  - a \mathcal{B})  = f( x + a \mathcal{B}) + a f'(x).
  \label{nice-1}
  \end{equation}
  Now replace $x$ by $x+ a \mathcal{U}$ and use the relation $- \mathcal{B} = \mathcal{B}+1$ to 
 convert the  left-hand side of \eqref{nice-1} to
  \begin{equation}
  f(x - a \mathcal{B} + a \mathcal{U}) = f( x + a( \mathcal{B}+1) + a \mathcal{U}) = f(x+a).
  \end{equation}
  \noindent
  The right-hand side of \eqref{nice-1} becomes
  \begin{equation}
  f( x + a \mathcal{B} + a \mathcal{U}) + a f'(x + a \mathcal{U}) = f(x)  + a f'(x + a \mathcal{U}).
  \end{equation}
  \noindent
  It follows that Theorem \ref{thm-general}, in the case $n=1$, is equivalent to the fundamental theorem of  Calculus
  \begin{equation}
  f(x+a) = f(x) + \int_{0}^{a} f'(x + u) \, du.
  \end{equation}
  \noindent
  This is now written in the form
  \begin{equation}
  \Delta_{a} f(x) = a f'( x + a \mathcal{U}),
  \end{equation}
  \noindent
  where $\Delta_{a}$ is the forward difference operator with step size $a$.  
  
  The proof of Theorem \ref{thm-general} for arbitrary $n$ follows from the method above and the elementary identity
  \begin{equation}
  \prod_{i=1}^{n} \Delta_{a_{i}} f(x) = a_{1} \cdots a_{n} f^{(n)}(x + a_{1}\mathcal{U}_{1} + \cdots + a_{n} \mathcal{U}_{n}).
  \end{equation}
  
 \section{Self-duality property for the Bernoulli-Barnes polynomials}
 \label{sec-selfdual}
 
 Given a sequence $\{ a_{k} \}$ define a new sequence $\{ a_{k}^{*} \}$ by the rule 
 \begin{equation}
 a_{n}^{*} = \sum_{k=0}^{n} \binom{n}{k} (-1)^{k} a_{k}.
 \end{equation}
 \noindent
 The inversion formula \cite[p. 192]{graham-1994a} gives 
 \begin{equation}
 a_{n} = \sum_{k=0}^{n} \binom{n}{k} (-1)^{k} a_{k}^{*}.
 \end{equation}
 \noindent
 The sequence $\{ a_{n}^{*} \}$ is called the dual of $\{ a_{n} \}$.  A sequence 
  is called {\textit{self-dual}} if it agrees with its dual.  Examples
 of self-dual sequences have been discussed in \cite{sunz-2000b,sunz-2003a}. For example, the fact that the sequence 
 $\{ (-1)^{n} B_{n} \}$ is self-dual is equivalent to the classical identity 
 \begin{equation}
 (-1)^{n} B_{n} = \sum_{k=0}^{n} \binom{n}{k} B_{k},
 \label{self-dual-ber}
 \end{equation}
 \noindent 
 which, expressed symbolically, is nothing but \eqref{bplus1}.  In  \cite{bayad-2014a} the authors prove the next
  result as Corollary $5.$. This is an extension of \eqref{self-dual-ber}  to the  Bernoulli-Barnes polynomials 
 and ask for a more direct proof.  Such a  proof is presented next. \\
 
 \noindent
 \begin{theorem}
 \label{thm-self-dual}
  Let 
$\mathbf{a} = (a_{1}, \cdots, a_{n})$ and $A = a_{1} + \cdots  + a_{n} \neq 0$. Then the sequence
 \begin{equation}
  p_{n} = (-1)^{n} A^{-n} B_{n}(\mathbf{a})
  \end{equation}
  \noindent
  is self-dual. 
  \end{theorem}
  \begin{proof}
  Observe that 
  \begin{eqnarray*}
  p_{n}^{*}  & = &  \sum_{k=0}^{n} \binom{n}{k} (-1)^{k} p_{k} \\
   & = & \sum_{k=0}^{n} \binom{n}{k}    A^{-k} ( \mathbf{a} \cdot \mathcal{B} )^{k} \\
   & = & \left( 1 + \frac{1}{A} \mathbf{a} \cdot \mathcal{B} \right)^{n} \\
   & = & A^{-n} ( A + \mathbf{a} \cdot \mathcal{B} )^{n} \\
   & = & A^{-n} \left( a_{1}(1 + \mathcal{B}_{1}) + \cdots + a_{n} (1+ \mathcal{B}_{n} )  \right)^{n} \\
   & = & A^{-n}  ( - \mathbf{a} \cdot \mathcal{B} )^{n} \\
   & = & (-1)^{n} A^{-n} B_{n}(\mathbf{a}) \\
   & = & p_{n}.
   \end{eqnarray*}
   \noindent
   This completes the proof.
   \end{proof}
   
     The authors of \cite{bayad-2014a} then ask for a  direct proof of the following symmetry formula. Such a 
     proof is presented next.

 \begin{theorem}
 Let $\mathbf{a} = (a_{1}, \cdots, a_{n}) \in \mathbb{R}^{n}$ with $\begin{displaystyle} 
 A = \sum_{k=1}^{n} a_{k} \neq 0\end{displaystyle}$. Then 
 for any integers $l, \, m \geq 0$,
 \begin{equation}
 \label{quest-1}
 (-1)^{m} \sum_{k=0}^{m} \binom{m}{k} A^{m-k} B_{l+k}(x;\mathbf{a}) = 
  (-1)^{l} \sum_{k=0}^{l} \binom{l}{k} A^{l-k} B_{m+k}(-x;\mathbf{a}),
 \end{equation}
 \noindent
 and 
 \begin{multline}
 \frac{(-1)^{m} }{m+l+2} \sum_{k=0}^{m} \binom{m+1}{k} (l + k + 1)A^{m+1-k} B_{l+k}(x;\mathbf{a}) \, \, +  \label{quest-2} \\
  \frac{(-1)^{l} }{m+l+2} \sum_{k=0}^{l} \binom{l+1}{k} (m+ k + 1)A^{l+1-k} B_{m+k}(x;\mathbf{a}) =  \\
  (-1)^{m+1} B_{l+m+1}(x; \mathbf{a}) + (-1)^{l+1} B_{l+m+1}(-x; \mathbf{a}).
  \end{multline}
  \end{theorem}
  \begin{proof}
  The left-hand side of \eqref{quest-1} can be written as 
  \begin{eqnarray*} 
   (-1)^{m} \sum_{k=0}^{m} \binom{m}{k} A^{m-k} B_{l+k}(x;\mathbf{a})  & = & 
   (-1)^{m} \sum_{k=0}^{m} \binom{m}{k} A^{m-k} ( x + \mathbf{a} \cdot \mathcal{B})^{l+k} \\
   & = & (-1)^{m} ( x + \mathbf{a} \cdot \mathcal{B})^{l} ( A + x + \mathbf{a} \cdot \mathcal{B})^{m} \\
   & = & (-1)^{m} ( x - A - \mathbf{a} \cdot \mathcal{B} )^{l} (x  - \mathbf{a} \cdot \mathcal{B} )^{m}
   \end{eqnarray*}
   \noindent
   using \eqref{bplus1}.  The right-hand side of \eqref{quest-1} is 
   \begin{equation}
   (-1)^{l} (-x + \mathbf{a} \cdot \mathcal{B})^{m} (-x + A + \mathbf{a} \cdot \mathcal{B})^{l} = 
   (-1)^{m} (x  - \mathbf{a} \cdot \mathcal{B})^{m} ( x- A - \mathbf{a} \cdot \mathcal{B})^{l}
   \end{equation}
   \noindent
   and this proves the  identity \eqref{quest-1}.  The second requested identity \eqref{quest-2} follows by differentiating 
   \eqref{quest-1}.
   \end{proof}
   
   \section{Some linear identities for the Bernoulli-Barnes numbers}
   \label{sec-linear}
   
   This section contains proofs of some linear recurrences for the Bernoulli-Barnes numbers by the symbolic method 
   discussed here.  The first result appears as Theorem 5.5 in \cite{bayad-2014a}. 
   
   \begin{theorem}
   Let $m \in \mathbb{N}, \, \mathbf{a} = (a_{1}, \cdots, a_{n})$ and $A = a_{1} + \cdots + a_{n}$. Then 
   \begin{equation}
   B_{2m+1}(\mathbf{a}) = - \frac{1}{2(m+1)} \sum_{k=0}^{m} \binom{m+1}{k} (m+k+1) A^{m+1-k} B_{m+k}(\mathbf{a})
   \label{ber-odd}
   \end{equation}
   \noindent
   and 
   \begin{eqnarray}
   \;\;\; B_{2m}(\mathbf{a}) & = &  - \frac{1}{(m+1)(2m+1)} \sum_{k=0}^{m-1} \binom{m+1}{k} (m+k+1)A^{m-k} B_{m+k}(\mathbf{a}) \\ 
   & & + \frac{ (2m)!}{A} \sum_{k=0}^{n-1} \sum_{|I|=k} \frac{B_{2m+1-n+k}(\mathbf{a}_{I})}{(2m+1-n+k)!}.
   \nonumber 
   \end{eqnarray}
   \end{theorem}
   \begin{proof}
   Start with the elementary identity 
   \begin{equation}
   -(m+1) y^{m} (2y^{m+1} - (x+y)^{m}(x+2y)) = \sum_{k=0}^{m} \binom{m+1}{k} (m+k+1) x^{m+1-k}y^{m+k}
   \end{equation}
   \noindent
   and denote the right-hand side by $f(y)$. Now use it with 
   $x = A = a_{1}+ \cdots + a_{n}$ and $y = a_{1} \mathcal{B}_{1} + \cdots + a_{n} \mathcal{B}_{n} = 
   \mathbf{a} \cdot \mathcal{B}$ to obtain
   \begin{eqnarray*}
   f( \mathcal{B})  & = & -(m+1) \left( 2 ( \mathbf{a} \cdot \mathcal{B})^{2m+1} - (A + \mathbf{a} \cdot \mathcal{B})^{m} 
   (A + 2 \mathbf{a} \cdot \mathcal{B}) (\mathbf{a} \cdot \mathcal{B})^{m} \right) \\
   & = & -(m+1) \left( 2 ( \mathbf{a}\cdot \mathcal{B})^{2m+1} - (A + \mathbf{a} \cdot \mathcal{B} )^{m+1} 
   (\mathbf{a} \cdot \mathcal{B})^{m} - (A + \mathbf{a} \cdot \mathcal{B})^{m} (\mathbf{a} \cdot \mathcal{B})^{m+1} 
   \right).
   \end{eqnarray*}
   \noindent
   Then $\mathcal{B} = - \mathcal{B} -1$ gives 
   \begin{equation}
   (A + \mathbf{a} \cdot \mathcal{B})^{m+1} ( \mathbf{a} \cdot \mathcal{B})^{m} =
   (- \mathbf{a} \cdot \mathcal{B})^{m+1} (-A - \mathbf{a} \cdot \mathcal{B})^{m} = 
   -(\mathbf{a} \mathcal{B})^{m+1} (A+ \mathbf{a} \cdot \mathcal{B})^{m}
   \end{equation}
   \noindent
   that can be written as 
   \begin{equation}
   (A+ \mathbf{a} \cdot \mathcal{B})^{m+1} (\mathbf{a} \cdot \mathcal{B})^{m} + 
   (\mathbf{a} \cdot \mathcal{B})^{m+1}  (A+ \mathbf{a} \cdot \mathcal{B})^{m} = 0.
   \end{equation}
   \noindent
   The proof follows from here.    
   
   The second formula contains a small  typo in the formulation given in \cite{bayad-2014a}.  To prove the corrected 
   formula, use \eqref{diff-51} with $x=0$ and $m$ replaced by $2m+1$ to obtain
   \begin{equation}
   -2B_{2m+1}( \mathbf{a}) = (2m+1)! \sum_{k=0}^{n-1} \sum_{|K|=k} \frac{B_{2m+1-n+k}({\mathbf{a}}_{K})}{(2m+1-n+k)!}.
   \end{equation}
   \noindent
   The expression \eqref{ber-odd} for $B_{2m+1}(\mathbf{a})$ just established now gives 
   \begin{multline*}
   \frac{(2m)!}{A} \sum_{k=0}^{n-1} \sum_{|K|=k} \frac{B_{2m+1-n+k}({\mathbf{a}}_{K})}{(2m+1-n+k)!} =  \\
   \frac{1}{(m+1)(2m+1)} \sum_{k=0}^{m} \binom{m+1}{k} (m+1+k) A^{m-k} B_{m+k}(\mathbf{a}).
   \end{multline*}
   \noindent
   Conclude with the observation  that the term corresponding to $k = m$ in the last 
   sum is $B_{2m}(\mathbf{a})$.  Solving for it gives 
   the stated expression.
   \end{proof}
   
   The identity presented next appears as Theorem 1.1 in  \cite{bayad-2014a}. 
   
   \begin{theorem}
   \label{main-1}
   For $n \geq 3, \, m \geq 1$ odd and $\mathbf{a} = (a_{1}, \cdots, a_{n}) \in \mathbb{R}^{n}$, 
   \begin{equation}
   \sum_{j=n-m} ^{n} \binom{n+j-4}{j-2} \frac{1}{(m-n+j)!} \sum_{|J|=j} B_{m-n+j}(\mathbf{a}_{J}) = 
   \begin{cases}
   \tfrac{1}{2} & \quad \text{ if } n = m = 3, \\
   0 & \quad \text{ otherwise},
   \end{cases}
   \end{equation}
   \noindent
   where the inner sum is over all subsets $J \subset \{ 1, \cdots, n \}$ of cardinality $j$. 
   \end{theorem}
   
   The proof presented next shows that Theorem \ref{main-1} is part of a general class of identities. The proof also 
   explains the appearance of the puzzling $\binom{n+j-4}{j-2}$. 
   
   \begin{theorem}
   \label{thm-general-method}
   Let $\{ \alpha_{j}^{(n)}: 1 \leq j \leq n \}$ be a sequence of numbers satisfying the palindromic condition 
   $\alpha_{n-j}^{(n)} = \alpha_{j}^{(n)}$ and let $f$ be an odd function. Then 
   \begin{equation}
   \label{sum-pali}
   \sum_{j=0}^{n} \alpha_{j}^{(n)} \sum_{|J|=j} f \left( ( \mathbf{a} \cdot \mathcal{B} )_{J} - 
   (\mathbf{a} \cdot \mathcal{B})_{J^{*}} \right)
   = 0,
   \end{equation}
   \noindent
   where $J^{*}$ is the complement of $J$ in $\{ 1, \cdots, n \}$. 
   \end{theorem}
   \begin{proof}
   Observe that 
   \begin{equation}
   ( \mathbf{a} \cdot \mathcal{B})_{J} -  ( \mathbf{a} \cdot \mathcal{B})_{J^{*}} = - \left(  ( \mathbf{a} \cdot \mathcal{B})_{J^{*}} -  ( \mathbf{a} \cdot \mathcal{B})_{J} \right)
   \end{equation}
   \noindent
   and so for each term 
   \begin{equation}
   \alpha_{j}^{(n)} f \left(   ( \mathbf{a} \cdot \mathcal{B})_{J} -  ( \mathbf{a} \cdot \mathcal{B})_{J^{*}}  \right)
   \end{equation}
   \noindent
   in the sum \eqref{sum-pali}, there is a corresponding term 
   \begin{equation}
   \alpha_{n-j}^{(n)} f \left(   ( \mathbf{a} \cdot \mathcal{B})_{J^{*}} -  ( \mathbf{a} \cdot \mathcal{B})_{J}  \right) = 
   \alpha_{j}^{(n)} f \left(   ( \mathbf{a} \cdot \mathcal{B})_{J^{*}} -  ( \mathbf{a} \cdot \mathcal{B})_{J}  \right).
   \end{equation}
   \noindent
   The fact that $f$ is an odd function implies
   \begin{equation}
   \alpha_{j}^{(n)} f \left(   ( \mathbf{a} \cdot \mathcal{B})_{J^{*}} -  ( \mathbf{a} \cdot \mathcal{B})_{J}  \right)  + 
   \alpha_{n-j}^{(n)} f \left(   ( \mathbf{a} \cdot \mathcal{B})_{J^{*}} -  ( \mathbf{a} \cdot \mathcal{B})_{J}  \right).
   \end{equation}
   \noindent 
   Hence the total sum over $j$ vanishes.
      \end{proof}

  \begin{example}
  Theorem \ref{main-1} corresponds to the choice 
  \begin{equation}
  \alpha_{j}^{(n)} = \begin{cases} 
  \binom{n-4}{j-2} & \quad \text{ if } 2 \leq j \leq n-2,  \\
  0 & \quad \text{ otherwise}. 
  \end{cases}
  \end{equation}
  To obtain this result start with the expansion
  \begin{equation}
  f( ( \mathbf{a} \cdot \mathcal{B})_{K}  -  (\mathbf{a} \cdot \mathcal{B})_{K^{*}}  ) = 
  \sum_{j=0}^{|K^{*}|} \sum_{|J| = j} |a_{J}| f^{(j)}( (\mathbf{a} \cdot \mathcal{B})_{J^{*}})
  \end{equation}
  \noindent
  and then 
  \begin{eqnarray*}
  \sum_{k=2}^{n-2} \sum_{|K|=k} \alpha_{k}^{(n)}
   f \left( ( \mathbf{a} \cdot \mathcal{B})_{K} - (\mathbf{a} \cdot \mathcal{B})_{K^{*}}   \right) & = & 
   \sum_{k=2}^{n-2} \sum_{|K|=k} \alpha_{k}^{(n)} \sum_{j=0}^{|K^{*}|} \sum_{|J|=j} |\mathbf{a}|_{J}
    f^{(j)}( (\mathbf{a} 
   \cdot \mathcal{B})_{J^{*}} )  \\
   & = & \sum_{j=0} \sum_{|J|=j} |\mathbf{a}_{J}|  f^{(j)}    \left(    ( \mathbf{a}  \cdot   \mathcal{B} )_{J^{*}} \right)  
    \sum_{k=2}^{n-2} \sum_{|K|=k} \alpha_{k}^{(n)}.
   \end{eqnarray*}
   \noindent
   Now
   \begin{equation}
   \sum_{|K| = k} 1 = \binom{n-j}{k}
   \end{equation}
   \noindent
   since there are $\binom{n-j}{k}$ subsets of $K$ of size $k$ in $\{ 1, \cdots, n \}$ that do not overlap with $J$. Hence 
   \begin{equation}
   \sum_{k=2}^{n-2} \sum_{|K|=k} \alpha_{j}^{(n)} = \sum_{k=2}^{n-2} \binom{n-4}{k-2} \binom{n-j}{n-j-k} 
   = \binom{2n-j-4}{n-j-2}
   \end{equation}
   by the Chu-Vandermonde identity \cite[p. 169]{graham-1994a}. This gives 
   \begin{equation*}
   \sum_{k=2}^{n-2} \sum_{|K|=k} \alpha_{k}^{(n)} f \left( ( \mathbf{a} \cdot \mathcal{B} )_{K} -  
   ( \mathbf{a} \cdot \mathcal{B} )_{K^{*}} \right) = 
   \sum_{j=0}^{n} \binom{2n-j-4}{n-j-4} \sum_{|J|=j} |\mathbf{a}|_{J} f^{(j)}  \left( ( \mathbf{a} \cdot \mathcal{B} )_{J^{*}} 
   \right).
   \end{equation*}
   \noindent
   The change of summation variable $j \mapsto n-j$ has the effect 
   \begin{equation}
   \binom{2n-j-4}{n-j-2}  \mapsto \binom{n+j-4}{j-2}
   \end{equation}
   \noindent 
   and this produces Theorem \ref{main-1} by taking $f(x) = x^{m}/m!$. 
  \end{example}
  
  \section{One final recurrence for the Bernoulli-Barnes numbers}
  \label{sec-final}
  
  Identities between generalized Bernoulli-Barnes numbers of different orders are rare in the literature. The symbolic 
  method used in this paper provides an efficient way to prove and generalize such identities, as shown in the cases 
  studied in the previous sections.   However, other techniques may compete favorably. This last section provides a 
  new occurrence of these identities and purely analytical proofs are provided.
  
 The exponential generating function for the Bernoulli-Barnes polynomials in the special case of parameter 
 $ \mathbf{1} = (1, \cdots, 1) \in \mathbb{C}^{n}$, is given in \eqref{gen-norlund1} by 
\begin{equation}
\sum_{j=0}^{\infty} B_{j}^{(n)}(x;\mathbf{1}) \frac{z^{j}}{j!} = e^{xz}  \frac{z^{n}}{(e^{z}-1)^{n}},
\label{gen-norlund}
\end{equation}
\noindent
where the parameter $n$ counts the length of $\mathbf{1} \in \mathbb{C}^{n}$.  Introduce the notation 
\begin{equation}
B_{j}^{(n)}(x) = B_{j}^{(n)}(x;\mathbf{1}),
\end{equation}
\noindent
and write \eqref{gen-norlund} as 
\begin{equation}
\sum_{j=0}^{\infty} B_{j}^{(n)}(x)  \frac{z^{j}}{j!} = e^{xz}  \frac{z^{n}}{(e^{z}-1)^{n}},
\label{gen-norlund-00}
\end{equation}
\noindent
This special case of Bernoulli-Barnes polynomials is also known as N\"{o}rlund polynomials. 
  
  A connection between hypergeometric function and these polynomials is now made explicit. The identity 
  \begin{equation}
  \pFq21{1 \, \,\,\,\, 1}{p+2}{z} = \frac{p+1}{z} \left[ 
  \sum_{\ell=0}^{p-1} \frac{1}{(p - \ell)} \left( \frac{z-1}{z} \right)^{\ell} - \left( \frac{z-1}{z} \right)^{p} \log (1-z) 
  \right]
  \end{equation}
  \noindent
  for the hypergeometric function 
  \begin{equation}
  \pFq21{1 \,\,\,\,\,  1}{p+2}{z} = \sum_{n=0}^{\infty} \frac{(1)_{n} (1)_{n}}{(p+2)_{n}} \frac{z^{n}}{n!} 
  \end{equation}
  \noindent
  can be found in \cite[$7.3.1.136$]{prudnikov-1990a}. The substitution $z \mapsto 1 - e^{z}$ gives 
  \begin{equation}
  \label{hyper-1}
  \pFq21{1 \, \,\,\,\, 1}{p+2}{1 - e^{z} } = (p+1) 
  \left[  \frac{z e^{pz} }{(e^{z}-1)^{p+1}} - 
  \sum_{\ell=0}^{p-1} \frac{1}{(p - \ell)}  \frac{e^{\ell z}}{(e^{z} - 1)^{\ell +1}}   \right].
  \end{equation}
  \noindent
  The terms in the sum above are now written in terms  of the Bernoulli-Barnes polynomial. To start, 
   \eqref{gen-norlund} gives 
  \begin{eqnarray*}
  \frac{ze^{pz}}{(e^{z}-1)^{p+1}} & = &  z^{-p} e^{pz} \left( \frac{z}{e^{z}-1} \right)^{p+1} = 
   z^{-p} \sum_{j=0}^{\infty} B_{j}^{(p+1)}(p) \frac{z^{j}}{j!}
    =  \sum_{j=-p}^{\infty} \frac{B_{j+p}^{(p+1)}(p)}{(j+p)!} z^{j}
  \end{eqnarray*}
  \noindent
 for the first term in \eqref{hyper-1}.  The second term in \eqref{hyper-1} can be  written as 
  \begin{equation}
  \sum_{\ell=0}^{p-1} \frac{1}{(p-\ell)} \frac{e^{\ell z}}{(e^{z}-1)^{\ell+1}}  = 
  \sum_{\ell=0}^{p-1} \frac{1}{(p-\ell)} \sum_{j=-\ell - 1}^{\infty} B_{j+\ell+1}^{(\ell+1)}(\ell) \frac{z^{j}}{(j+\ell+1)!}.
  \end{equation}
  \noindent
  Since the hypergeometric function is analytic at $z=0$, the coefficients of negative powers on the 
  right-hand side of  \eqref{hyper-1} must 
  vanish. This  leads, for $-p \leq j \leq -1$, to the identity 
  \begin{equation}
  \frac{B_{j+p}^{(p+1)}(p) }{(j+p)!} = 
  \sum_{\ell = -j-1}^{p-1} \frac{1}{p -\ell} \frac{B_{j+\ell+1}^{(\ell+1)}( \ell)}{(j+ \ell + 1)!}.
  \end{equation}
  \noindent
  A shift in the index and denoting $j+p$ by $r$ produces the final statement. 
  
  \begin{theorem}
  Let $0 \leq r \leq p-1$. Then 
  \begin{equation}
  \frac{B_{r}^{(p+1)}(p) }{r!}  = \sum_{k=1}^{r+1} \frac{1}{k} \frac{B_{r+1-k}^{(p+1-k)}(p-k)}{(r+1-k)!},
  \end{equation}
  \noindent
  or
  \begin{equation}
  \frac{B_{r}^{(p+1)}(p)}{r!} - \frac{B_{r}^{(p)}(p-1)}{r!}  = 
  \sum_{k=1}^{r} \frac{1}{(k+1)}  \frac{B_{r-k}^{(p-k)}(p+1-k)}{(r-k)!}.
  \end{equation}
  \end{theorem}

 \medskip

\noindent
\textbf{Acknowledgments}. The second author acknowledges the partial 
support of NSF-DMS 1112656.  The first author is a graduate student partially funded by this grant. The work of the last  author was partially funded by the iCODE Institute,
a research project of the Idex Paris-Saclay.

%\bibliography{../../../AllRef/official}
%\bibliographystyle{plain}

%\end{document}

\end{document}